\documentclass{scrartcl}
\usepackage[dvipsnames]{xcolor}
\usepackage[english]{babel}
\usepackage[utf8]{inputenc}
\usepackage{mathpazo}
\usepackage{amsmath}
\usepackage{amssymb}
\usepackage{amsthm}
\usepackage{graphicx}
\usepackage[caption=false]{subfig}
\usepackage[top=1.25in]{geometry}
\usepackage{enumitem}
\usepackage[hyphens]{url}
\usepackage[hidelinks]{hyperref}
\usepackage[capitalise]{cleveref}
\usepackage{mathrsfs}
\usepackage{caption}
\usepackage{tabularx}

\usepackage{mathtools} 

\newcommand{\mb}[1]{\mathbb{#1}}
\newcommand{\mbbr}[0]{\mathbb{R}}
\newcommand{\mbbn}[0]{\mathbb{N}}
\newcommand{\weaklyto}[2]{\ensuremath{#1 \xrightarrow{w} #2}}

\newcommand{\Pb}[0]{\mathbb{P}}

\newcommand{\pbsomega}[0]{\ensuremath{(\Omega, \mathcal{F}, P})}

\newcommand{\asto}[1]{\ensuremath{\;\;\textnormal{as}\;\;#1 \to \infty}}

\newcommand{\du}{\, \mathrm{d}}

\DeclarePairedDelimiter{\abs}{\lvert}{\rvert}
\newtheorem{lemma}{Lemma}
\newtheorem{theorem}{Theorem}

\theoremstyle{definition}
\newtheorem{remark}{Remark}

\newtheorem{example}{Example}
\newtheorem{assumption}{Assumption}

\crefdefaultlabelformat{#2\textup{#1}#3}
\Crefname{assumption}{Assumption}{Assumptions}
\Crefrangeformat{assumption}{Assumptions #3#1#4--#5#2#6}

\newcommand{\keywords}[1]
{
	{\small
		\textbf{Key words.} {#1}
		\\
	}
}

\newcommand{\amssubject}[1]
{
	{\small
		\textbf{AMS subject classifications.} {#1}
	}
}

\renewcommand{\abstract}[1]
{
	{\small
		\textbf{Abstract.} {#1}
		\\
	}
}

\title{Randomized quasi-Monte Carlo methods for risk-averse stochastic optimization}
\author{Olena Melnikov\thanks{H.\ Milton Stewart School of Industrial and Systems Engineering, Georgia Institute of Technology, Atlanta, Georgia 30332 (omelnikov6@gatech.edu,johannes.milz@gatech.edu).}
\and Johannes Milz\footnotemark[1]
 }
\date{\today}

\begin{document}

	\maketitle

	\abstract{
    We establish epigraphical and uniform laws of large numbers for sample-based approximations of law invariant risk functionals. These sample-based approximation schemes include Monte Carlo (MC) and certain randomized quasi-Monte Carlo integration (RQMC)
    methods, such as scrambled net integration. Our results can be applied to the approximation of risk-averse stochastic programs and risk-averse stochastic variational inequalities. Our numerical simulations empirically demonstrate that RQMC approaches based on scrambled Sobol' sequences can yield smaller bias and root mean square error than MC methods for risk-averse optimization.
    }

	\keywords{quasi--Monte Carlo, randomized quasi--Monte Carlo, risk-averse optimization, sample average approximation}

    \amssubject{90C15, 90C59, 65C05}

	\section{Introduction}

We consider the risk-averse stochastic optimization problem
\begin{equation}\label{eq:op}
    \min_{x \in X} \rho(F(x, \xi)),
\end{equation}
where \(X\) is a potentially infinite dimensional decision space, \(\Omega\) and \(\Xi\) are sample spaces,
and \(\xi: \Omega \to \Xi\) is a random element. The real-valued
function \(F\) is defined on \(X \times \Xi\) and the real-valued,
law invariant risk measure $\rho$ is defined
on the Lebesgue space $\mathcal{L}^p(\Omega)$
with $p \in [1, +\infty)$.

Two common challenges arise when trying to solve \eqref{eq:op} in practice.
First, knowledge of the distribution of $\xi$ may be incomplete.
For instance, only samples of $\xi$ may be available.
Second, the approximation of high-dimensional integrals is computationally intractable.
One way to address these issues is to construct
sample-based approximations of \eqref{eq:op}
and solve them instead.
Sample average approximation---a Monte Carlo (MC) sample-based approach---is a common
approach used to generate consistent approximations to \eqref{eq:op} \cite{Kleywegt2001,shapiropaper}. However,
randomized quasi-Monte Carlo (RQMC) can have superior performance
in terms of accuracy and variance reduction. Consequently, there is interest in
analyzing the effectiveness of methods such as RQMC sampling in the context of stochastic
programming.

For the risk-neutral case, the epiconsistency of RQMC-based approximations of \eqref{eq:op} is demonstrated in \cite{koivu2004variance}. The numerical experiments in \cite{koivu2004variance} suggest favorable
performance of the RQMC sample-based approach compared to MC sampling when applied
to risk-neutral stochastic programs.
A similar analysis has been conducted for approximation
schemes based on quasi-Monte Carlo (QMC) sampling in \cite{teemukoivu2005}. Furthermore, \cite{Hess2019} provides consistency
results for a general class of approximation schemes in the risk-neutral case. These
schemes approximate the distribution of \(\xi\) using auxiliary distributions. The
auxiliary distributions are required to meaningfully approximate the distribution of
\(\xi\) with respect to specific integrals. The authors of \cite{Hess2019}  demonstrate that these schemes yield epiconsistent approximations.

Some problems cannot be appropriately modeled with risk-neutral formulations. For instance, portfolio optimization often requires a risk-averse model to avoid
potentially detrimental losses.
For the risk-averse case, the epiconsistency the MC sample-based approach
has been demonstrated in \cite{shapiropaper}.
Moreover, a central limit theorem for MC sample-based composite risk functionals is established in
\cite{Dentcheva2017}.

Our main contribution lies in demonstrating the almost sure epiconvergence and uniform convergence of a class of sample-based approximation schemes for
composite risk functionals beyond MC sample-based approaches. In more detail, let \(\Theta\) be a sample space and for each $x \in X$,
let $F_x(\cdot)$ denote the measurable map
$\xi \mapsto F(x, \xi)$ on $\Xi$.
For a random variable \(Z: \Xi \to \mathbb{R}\) and a probability measure \(\nu\) on \(\Xi\),
let \(\nu \circ Z^{-1}\) denote the corresponding pushforward measure.
Roughly speaking and omitting technical considerations,
for suitable random sequences of probability measures
\((\nu_{N, \theta})_{N \in \mathbb{N}}\) indexed by \(\theta \in \Theta\), we demonstrate the epiconvergence and uniform
convergence of
\begin{equation}\label{eq:approxproblem}
    x \mapsto \rho(\nu_{N, \theta}\circ F_x^{-1})
\end{equation}
towards
\begin{equation}
    \label{eq:objective}
    x \mapsto \rho(F_x(\xi)).
\end{equation}
In the spirit of a strong law of large numbers (SLLN),
we require the sequence of
probability measures
\((\nu_{N, \theta})_{N \in \mathbb{N}}\)
to yield consistent approximations
of risk functionals as $N \to \infty$.
More specifically, we demonstrate that as long as for all random variables
$Z \colon \Xi \to \mathbb{R}$ such that \(Z \circ \xi \in \mathcal{L}^q(\Omega)\)
with $q \geq p$, and
\begin{equation}\label{eq:conditionforscheme}
    \rho(\nu_{N, \theta} \circ Z^{-1}) \to \rho(Z(\xi)) \quad
    \quad \text{as} \quad  N \to \infty \quad \text{for almost all} \quad \theta \in \Theta,
\end{equation}
the sample-based composite risk functions \eqref{eq:approxproblem}
yield epiconsistent approximations to \eqref{eq:objective}.
The condition \eqref{eq:conditionforscheme}
is canonically satisfied for MC sample-based approximations.
For these approximations, our epiconsistency
results coincide with those in \cite{shapiropaper}.
We further demonstrate that \eqref{eq:conditionforscheme}
is satisfied for certain RQMC-based schemes, notably those given by
scrambled digital nets.
Our verification leverages
the SLLN recently established in
\cite{Owen2021}. When combined with standard properties
of epiconvergence
and compactness of the decision space $X$,
our epigraphical law of large numbers provides conditions
sufficient for almost sure consistency of the optimal value of
\begin{align*}
    \min_{x \in X}\, \rho(\nu_{N, \theta}\circ F_x^{-1})
\end{align*}
towards that of
\eqref{eq:op}.

Furthermore, we provide conditions sufficient for the uniform convergence
of \eqref{eq:approxproblem} towards \eqref{eq:objective}. These sufficient conditions
are also based on a SLLN, but are more stringent than
\eqref{eq:conditionforscheme}.
For MC sample-based approximations, this uniform law of large numbers
coincides with the standard one for real-valued Carath\'eodory functions. We use this uniform SLLN to demonstrate consistency of the sample-based approximations to risk-averse stochastic variational inequalities of the form
\begin{equation}\label{eq:intro:simplevi}
     0 \in \Big[\rho_{1}(F^1_x(\xi)), \ldots, \rho_{r}(F^r_x(\xi))\Big]
        + {\partial\psi(x)},
\end{equation}
where $\psi \colon \mathbb{R}^r \to (-\infty,+\infty]$ is proper, closed,
and convex with compact domain $X$, and
$\rho_j$, $j=1, \ldots, r$, are real-valued, law invariant risk measures
and $F_x^j$,  $j=1, \ldots, r$, are maps with properties related
to those of $F_x$. Here $\partial \psi(x)$ is the subdifferential of
$\psi$ at $x \in \mathbb{R}^r$.
An instance of \eqref{eq:intro:simplevi}
is considered in \cite{Cherukuri2024},
where all $\rho_j$ equal the conditional value-at-risk,
and $\psi$ is the indicator function of a nonempty,
convex, compact set $X$.
The variational inequality in \eqref{eq:intro:simplevi}
is an instance of the
distributionally robust stochastic variational inequalities considered
in \cite{Sun2023}.

\paragraph{Overview}
After introducing notation in \Cref{sec:notation}, we discuss conditions sufficient for
consistent approximations of law invariant risk functionals
and provide examples in \Cref{sec:consistentapproximations}.
In \Cref{sec:epiconvergence}, we establish conditions sufficient for
the almost sure epiconvergence  of composite risk functionals.
Subsequently, we give conditions sufficient for a uniform law of large numbers.
We apply this result to risk-averse stochastic variational inequalities.
Finally, we provide numerical experiments comparing MC
and RQMC sample-based methods for risk-averse stochastic programs in terms of the
root mean square error (RMSE) and bias in \Cref{sec:numericalexperiements}.

\section{Notation and Preliminaries}
\label{sec:notation}

Throughout the text, let $(\Omega,\mathcal{F}, P)$ be a probability space, let $\Xi$ be a separable metric space equipped with its Borel $\sigma$-algebra,
let $\xi \colon \Omega \to \Xi$ be a random element, let
$p \in [1,+\infty)$ and $\varepsilon \geq 0$ be scalars. We impose additional
assumptions on these objects as needed. Moreover, let $\lambda$ denote the Lebesgue measure on $[0,1]^d$.
With a slight abuse of notation, we use $\xi$ to denote both the random element introduced above and deterministic elements in $\Xi$.

Next, we introduce basic terminology and notation.
We denote the uniform distribution on a set $C \subset \mathbb{R}^d$ by $\text{Uniform}(C)$.
We call $(\Omega, \mathcal{F}, P)$ \emph{nonatomic} if there
exists a random variable $U: (\Omega, \mathcal{F}, P) \to (0,1)$ such
that $U \sim \text{Uniform}(0,1)$. We denote the space of $p$-integrable random variables on $(\Omega, \mathcal{F}, P)$ by $\mathcal{L}^p(\Omega, \mathcal{F}, P)$, or simply by $\mathcal{L}^p(\Omega)$.
A functional $\rho: \mathcal{L}^p\pbsomega \to (-\infty, +\infty]$ is called
\emph{law invariant} (with respect to $P$) if for each $Y,Z \in \mathcal{L}^p\pbsomega $
with the same distribution, we have $\rho(Y) = \rho(Z)$
\cite[Def.\ 6.26]{Shapiro2021}. A risk measure $\rho  \colon \mathcal{L}^p(\Omega, \mathcal{F}, P) \to (-\infty, +\infty]$ is said to be \emph{convex} if it is a convex, monotone, and translation equivariant function \cite[Def.\ 6.4]{Shapiro2021}.
Let $V$ be a nonempty set. The indicator function
$\mb{1}_{V_0}\colon V \to \mbbr$  of a subset
$V_0 \subset V$ is defined by $\mb{1}_{V_0}(v) = 1$ if
$v \in V_0$ and  $\mb{1}_{V_0}(v) = 0$ otherwise.
Let $\mu, \mu_1, \mu_2, \ldots$ be probability measures,
and $Z, Z_1, Z_2, \ldots$ be random variables.
We denote the weak convergence of the
sequence $(\mu_n)_{n \in \mathbb{N}}$ to $\mu$,
and respectively, $(Z_n)_{n \in \mathbb{N}}$ to $Z$,
by $\weaklyto{\mu_n}{\mu}$ and $\weaklyto{Z_n}{Z}$
as $n \to \infty$.  If not specified otherwise,
metric spaces $\Upsilon$ are equipped with their Borel $\sigma$-algebra $\mathcal{B}(\Upsilon)$. For a metric space $X$, a function
$G\colon X \times \Upsilon \to \mathbb{R}$ is called \emph{a Carath\'eodory function} if
$G(x, \cdot)$ is $\mathcal{B}(\Upsilon)$-$\mathcal{B}(\mathbb{R})$-measurable
for each $x \in X$,
and $G(\cdot, \upsilon)$ is continuous for every $\upsilon \in \Upsilon$.
A function $G\colon X \times \Upsilon \to (-\infty,+\infty]$ is called \emph{random lower semicontinuous} if
the epigraphical multifunction
$\upsilon \mapsto \{(x, t) \in X \times \mathbb{R} \colon t \geq G(x,\upsilon)\}$ is closed-valued
and measurable.
For functions $f$, $f_n \colon X \to (-\infty, +\infty]$  \emph{epiconvergence}
of $(f_n)_{n \in \mathbb{N}}$ to $f$ as $n \to \infty$ means that
(i) for all sequences $(x_n)_{n \in \mathbb{N}}$
and points $x \in X$ with $x_n \to x$ as $n \to \infty$,
$\liminf_{n \to \infty}\, f_n(x_n) \geq f(x)$,
and (ii) for each $x \in X$, there exists a sequence
$(x_n)_{n \in \mathbb{N}}$ with $x_n \to x$ as $n \to \infty$
such that $\limsup_{n \to \infty}\, f_n(x_n) \leq  f(x)$.
For a real-valued random variable $Z$, let $H[Z]: (0,1) \to \mathbb{R}$
denote the quantile function associated with the distribution of $Z$. For a
probability distribution $Q$ on $\mathbb{R}$, let $H_Q(\cdot)$ indicate the corresponding quantile function.

\begin{remark}\label{lawinvariance}
If $(\Omega, \mathcal{F}, P)$ is nonatomic, then for each probability  distribution
$Q$ on $\mbbr$ there exists a random variable defined
on $(\Omega, \mathcal{F}, P)$ with distribution $Q$.
Indeed, by definition, nonatomicness guarantees the existence of
a random variable $U$ defined on $(\Omega, \mathcal{F},P)$
with $U\sim\text{Uniform}(0,1)$. Since $H_Q$ has distribution $Q$,
we obtain
\begin{equation*}
    H_{Q}(U)\sim Q.
\end{equation*}
\end{remark}

The above remark justifies viewing law invariant risk measures defined
on $\mathcal{L}^p(\Omega, \mathcal{F}, P)$
as functions of probability distributions on $\mbbr$ with finite $p$th moment,
provided that $(\Omega, \mathcal{F}, P)$ is nonatomic.

\section{Consistent Approximations of Law Invariant Functionals}
\label{sec:consistentapproximations}
Empirical distributions have been shown to yield consistent
approximations of risk functionals \cite{Bartl2022,Rachev1998,shapiropaper}.
The main result of this section,
\Cref{thm:consistency} presented below, generalizes this result to a broader class of distributions.
Notably, we use \Cref{thm:consistency} to demonstrate that MC and scrambled net integration
consistently approximate risk functionals defined on Lebesgue spaces.

\begin{theorem}[Consistent Approximations of Law Invariant Functionals]\label{thm:consistency}
Let $\pbsomega$ be nonatomic.
Consider a law invariant, continuous functional
$\rho: \mathcal{L}^p\pbsomega \to \mbbr$.
Let $\mu, \mu_1, \mu_2, \ldots$ be probability measures on $\mbbr$, and let
$\int_{\mbbr} \abs{t}^p \, \du \mu(t) < \infty$.
    Further, suppose that
    \begin{equation}
    \label{eq:weakmoment}
        \weaklyto{\mu_N}{\mu}\quad\text{and}\quad \int_{\mbbr} \abs{t}^p \du\mu_N(t)
        \to \int_{\mbbr} \abs{t}^p \du\mu(t)
    \end{equation}
    as $N \to \infty$. Then
    \begin{equation*}
        \rho(\mu_N) \to \rho(\mu)\asto{N}.
    \end{equation*}
\end{theorem}
\begin{proof}
Theorem~1.4.1 in \cite{Rachev1998} yields
    \begin{equation*}
        \int_{(0,1)} \abs{H_{\mu_N}(t)-H_{\mu}(t)}^p \du t \to0\asto{N}.
    \end{equation*}
Thus, for $U\sim \text{Uniform}(0,1)$ defined on $(\Omega, \mathcal{F},P)$, $H_{\mu_N}(U)$ converges to $H_{\mu}(U)$ in $\mathcal{L}^p(\Omega)$. Law
invariance and continuity of $\rho$ yield the result (see \Cref{lawinvariance}).
\end{proof}

Empirical means defined by standard
MC sampling satisfy a SLLN
and the corresponding
empirical distributions converge weakly, under mild assumptions.
These facts can be used to verify the convergence
statements in \eqref{eq:weakmoment} for MC sampling. More generally, the conditions in \eqref{eq:weakmoment} hold for all measures generated by a sampling method that satisfies a SLLN-type condition. We formulate this SLLN-type condition in the assumption below and demonstrate how it can
be used to verify \eqref{eq:weakmoment} in the next lemma.

\begin{assumption}\label{assumption:slln}
Let $(\xi^i)_{i \in \mbbn}$ be
a sequence of $\Xi$-valued random elements
defined on a probability space $(\Theta, \mathcal{A}, \Pb)$. For each random variable $\vartheta: \Xi \to \mbbr$ with $\vartheta \circ \xi \in
\mathcal{L}^{1 + \varepsilon/p}(\Omega)$, we have
\begin{equation*}
    \frac{1}{N}\sum_{i=1}^N \vartheta(\xi^i(\theta)) \to \int_{\Omega} \vartheta\circ \xi \du P \asto{N}
    \quad\text{for }\Pb\text{-almost all }\theta \in \Theta.
\end{equation*}
\end{assumption}

Let \Cref{assumption:slln} hold.
For $\theta \in \Theta$ and $N \in \mathbb{N}$, let us define the probability measure $\nu_{N, \theta}$ on $\Xi$ by
\begin{equation}\label{eq:nuNtheta}
    \nu_{N, \theta}(B) \coloneqq \frac{1}{N} \sum_{i=1}^N \mathbb{1}_{\xi^i(\theta)}(B)
\end{equation}
for each $B \in \mathcal{B}(\Xi)$.
We show that \Cref{assumption:slln}
allows us to verify the conditions in \eqref{eq:weakmoment}.
In a later section, we also reuse \Cref{assumption:slln} to formulate our uniform law of large numbers.

\begin{lemma}
\label{lem:weakmoment}
Let \Cref{assumption:slln} hold, and
let $Z \colon \Xi \to \mathbb{R}$
be a random variable with $Z \circ \xi \in \mathcal{L}^{p+\varepsilon}(\Omega)$.
Let
$\mu_{N,\theta} \coloneqq \nu_{N, \theta} \circ Z^{-1}$ and let
$\mu$ be the distribution of $Z \circ \xi$.
Then $\weaklyto{\mu_{N,\theta}}{\mu}$
and
$\int_{\mbbr} \abs{t}^p \du\mu_{N,\theta}(t)
        \to \int_{\mbbr} \abs{t}^p\du\mu(t)$
for $\mathbb{P}$-almost all $\theta \in \Theta$.
\end{lemma}

\begin{proof}
First, we show that
$\weaklyto{\mu_{N, \theta}}{\mu}$
as $N \to \infty$ for $\Pb$-almost all $\theta \in
\Theta$. We denote the countable set of continuous, bounded
functions $W: \mbbr \to \mbbr$ given in
Corollary~2.2.6 in \cite{Bogachev2018} by $\mathcal{W}$.
For each $W \in \mathcal{W}$, $(W \circ Z)\circ \xi \in \mathcal{L}^{\infty}(\Omega)$. Thus,
for each $W \in \mathcal{W}$, \Cref{assumption:slln} yields the existence of
a $\Pb$-full set $\Theta_W$ such that
\begin{equation}\label{eq:jul102024}
    \int_{\Xi} W \circ Z\du\nu_{N, \theta} =
    \frac{1}{N} \sum_{i=1}^N W(Z(\xi^i(\theta))\to \int_{\Omega} (W \circ Z)\circ \xi \du P \asto{N}
\end{equation}
for each $\theta \in \Theta_W$. The set
$\Theta_1 \coloneqq \bigcap_{W \in \mathcal{W}}\Theta_W$
is $\Pb$-full since it is a countable intersection of $\Pb$-full sets.
Now, for each $\theta \in \Theta_1$ and  $W \in \mathcal{W}$, a change of variables on both sides of \eqref{eq:jul102024} yields
\begin{equation*}
      \int_{\mbbr} W \du\mu_{N, \theta}(t)
      \to \int_{\mbbr} W \du \mu(t) \asto{N}.
\end{equation*}
Combined with Corollary~2.2.6 in
\cite{Bogachev2018}, we have
$\weaklyto{\mu_{N, \theta}}{\mu}$
as $N \to \infty$ for all $\theta \in \Theta_1$.
Since $Z \circ \xi \in \mathcal{L}^{p+\varepsilon}(\Omega)$,
we have $|Z(\xi)|^p \in \mathcal{L}^{1+\varepsilon/p}(\Omega)$.
Using \Cref{assumption:slln},
we also obtain the existence of a $\Pb$-full set $\Theta_2$ such that
\begin{equation*}
\int_{\mbbr} \abs{t}^p \du\mu_{N,\theta}(t)
=
    \int_{\Xi} \abs{Z}^p\du\nu_{N,\theta} \to \int_{\Xi}\abs{Z(\xi)}^p\du P
    =
    \int_{\mbbr} \abs{t}^p \du\mu(t)
    \asto{N}
\end{equation*}
for all $\theta \in \Theta_2$.
\end{proof}

\subsection{Examples}
\label{subsec:examples}
We use \Cref{thm:consistency} to derive consistency results
for approximations of risk functionals generated by MC and
QMC sampling, scrambled net integration, and
Latin hypercube sampling.
Importantly, these examples
serve as justification for a pivotal assumption, \Cref{ass:consistency},
used to formulate our epigraphical law of large numbers in later sections.

For the examples below we require the following assumption.
\begin{assumption}\label{ass:examples}
The probability space $(\Omega, \mathcal{F},P)$ is nonatomic, and
the risk measure $\rho:\mathcal{L}^p(\Omega, \mathcal{F},P) \to \mbbr$
is law invariant and continuous.
\end{assumption}

First, we demonstrate that \Cref{thm:consistency}
is compatible with classical MC sampling.
\begin{example}[MC-based measures]\label{exp:mcmeasures}
Let $Y  \in\mathcal{L}^p(\Omega,\mathcal{F}, P)$.
Let $Y_1, Y_2, \ldots $ be independent, identically
distributed (iid) random variables defined on a
probability space $(\Theta, \mathcal{A}, \Pb)$ with $Y_1\sim Y $.
We define the empirical distribution corresponding to
$Y_1, Y_2, \ldots, Y_N$ by
\begin{equation*}
    \pi_{N, \theta}(B) \coloneqq \frac{1}{N} \sum_{i=1}^N \mb{1}_{Y_i(\theta)}(B)
\end{equation*}
for each $B \in \mathcal{B}(\mbbr)$ and $\theta \in \Theta$.
Let \Cref{ass:examples} be satisfied.
Using the SLLN and the almost sure weak
convergence of empirical distributions, we find that
\begin{equation*}
            \rho(\pi_{N, \theta}) \to \rho(Y)\asto{N}
            \quad\text{for } \Pb\text{-almost all } \theta \in \Theta.
\end{equation*}
This convergence statement has already been established in
\cite[pp.\ 32--33]{Rachev1998}, Theorem~2.1 in
\cite{shapiropaper},
and Theorem~9.65 in \cite{Shapiro2021}.
\end{example}

It is worth noting that
MC sampling can also be examined using \Cref{lem:weakmoment}.
The resulting formulation is consistent with that of
\Cref{ass:consistency} in the next section.

\begin{example}[MC-based measures, continued]\label{exp:mcmeasurescontinued}
Let
$(\xi^i)_{i \in \mathbb{N}}$
be a sequence of iid random elements on $(\Theta, \mathcal{A}, \Pb)$
with $\xi^1 \sim \xi$.
Let \Cref{ass:examples} be satisfied.
Owing to the SLLN, \Cref{assumption:slln} is met
for $\varepsilon = 0$.
Now, \Cref{lem:weakmoment,thm:consistency} ensure
that for each random variable $Z \colon \Xi \to \mbbr$ with
$Z \circ \xi \in \mathcal{L}^{p}(\Omega)$, we have
\begin{equation*}
    \rho(\nu_{N, \theta} \circ Z^{-1}) \to \rho(Z(\xi))\asto{N}
    \quad\text{for } \Pb\text{-almost all } \theta \in \Theta,
\end{equation*}
where $\nu_{N, \theta}$ as in \eqref{eq:nuNtheta} with iid $(\xi^i)_{i \in \mbbn}$.
\end{example}

The next example is in the spirit of QMC methods.
QMC methods typically approximate the Lebesgue measure on
the unit cube using discrete measures
defined by deterministically generated points,
also referred to as low-discrepancy sequences.
Roughly speaking, low-discrepancy sequences are
designed to cover the integration domain evenly.
As a consequence, these deterministic measures weakly
converge to the Lebesgue measure.
While QMC methods avoid the potentially unfavorable clustering of
points that can occur with the vanilla MC method, they typically require
some smoothness assumptions on the integrand. We demonstrate
a consistency result for weakly convergent measures in the context
of risk measures, provided that a  smoothness condition is satisfied.


\begin{example}[Weakly convergent measures]\label{exp:weakconvergence}
Let $\nu$ be a probability measure on
$(\Xi, \mathcal{B}(\Xi))$. Suppose $Z$ is a real, $\nu$-almost surely continuous, and
bounded random variable defined on $(\Xi, \mathcal{B}(\Xi), \nu)$. Let $\nu_1,\nu_2,
\ldots$ be probability measures with
\begin{equation*}
    \weaklyto{\nu_N}{\nu}\asto{N}.
\end{equation*}
We have
\begin{equation*}
        \rho(\nu_{N}\circ Z^{-1}) \to \rho(\nu \circ Z^{-1})\asto{N}.
\end{equation*}
Let us verify this assertion.
Owing to the Continuous Mapping Theorem (see Theorem~13.25 in \cite{Klenke2014}), we
have
$$
    \weaklyto{\nu_N\circ Z^{-1}}{\nu \circ Z^{-1}}\asto{N}.
$$
Since $\abs{Z}^p$ is $\nu$-almost surely continuous and bounded, we obtain
\begin{equation*}
    \lim_{N \to \infty}\int_{\Xi}\abs{Z_N}^p \du\nu_N = \int_{\Xi}\abs{Z}^p \du\nu
\end{equation*}
(see Corollary~2.2.10 in \cite{Bogachev2018}). Now,
\Cref{thm:consistency} implies the assertion.
\end{example}

Many RQMC methods randomize
deterministic, low-discrepancy points in a way that
guarantees uniform distribution of the randomized points
while maintaining their low-discrepancy property. In the next
example, we consider measures that are obtained by randomizing
digital nets on $[0, 1]^d$ with a uniform nested scramble
\cite{Owen2021}. Notably, the uniform nested scramble can be used to randomize Sobol' sequences, which are $(t,d)$-sequences in base $2$ (see Definition 3.2 in \cite{Owen2021}).

\begin{example}[Scrambled Net Integration]\label{exmp:rqmc}
Let us give details on the RQMC methods considered in this example.
Let $(y^i)_{i \in \mbbn}\subset [0,1]^d$ be
a $(t,d)$-sequence in integer base $b \geq 2$ (see Definition~3.2 in \cite{Owen2021}), where $t \geq 0$ is an integer.
Moreover, suppose that the maximum gain coefficient of
the sequence is bounded by a finite constant.
Let the sequence $(\xi^i)_{i \in \mbbn}$ of random
vectors be obtained through randomization of the
sequence
$(y^i)_{i \in \mbbn}$ using the method
described in \cite{owen1995randomly}.
The random vectors $\xi^1, \xi^2, \ldots$ are defined on a probability space $(\Theta, \mathcal{A}, \mathbb{P})$,
and uniformly distributed on $[0,1]^d$, but are not independent. Let $q \in (1, +\infty]$.
Theorem~5.3 in \cite{Owen2021} ensures
that for each $\vartheta \in \mathcal{L}^q[0,1]^d$, we have
\begin{equation*}
  \frac{1}{N}\sum_{i=1}^N \vartheta(\xi^i (\theta)) \to \int_{[0,1]^d} \vartheta \du\lambda\asto{N}
  \quad\text{for } \Pb\text{-almost all } \theta \in \Theta.
\end{equation*}

Let  the probability measure $\nu_{N, \theta}$ be as in \eqref{eq:nuNtheta}, with the sequence $(\xi^i)_{i \in \mbbn}$ as described above.
We consider $\varepsilon > 0$ and
 $q = 1 + \varepsilon/p$.
Let $Z \in \mathcal{L}^{p + \varepsilon}[0,1]^d$ and let
$\xi \sim \text{Uniform}[0,1]^d$. Then, the
above SLLN, and
\Cref{lem:weakmoment,thm:consistency}
with $\Xi = [0,1]^d$ yield
\begin{equation*}
    \rho(\nu_{N, \theta} \circ Z^{-1}) \to \rho(Z(\xi)) \asto{N}
    \quad\text{for } \Pb\text{-almost all } \theta \in \Theta.
\end{equation*}
\end{example}

Next, we demonstrate consistent estimation of risk functionals
via Latin hypercube sampling, which is
a popular variance reduction scheme
used in stochastic programming.

\begin{example}[{Latin hypercube sampling}]
    Let the sequence $(\xi^i)_{i \in \mbbn} d$ of random
    vectors be obtained through Latin hypercube
    sampling as described in
    Section~1 in \cite{Loh1996}
    (see also Section~5.5.1 in \cite{Shapiro2021}).
    The random vectors $\xi^1, \xi^2, \ldots$ are defined on a probability space $(\Theta, \mathcal{A}, \mathbb{P})$,  and uniformly distributed on $[0,1]^d$, but are not independent.
    Then,    Theorem~3 in \cite{Loh1996}
    ensures that for each $\vartheta \in \mathcal{L}^2[0,1]^d$, we have
    \begin{equation*}
      \frac{1}{N}\sum_{i=1}^N \vartheta(\xi^i (\theta)) \to \int_{[0,1]^d} \vartheta \du\lambda\asto{N}
      \quad\text{for } \Pb\text{-almost all } \theta \in \Theta.
    \end{equation*}
    Let \Cref{ass:examples} be satisfied.
    Moreover,
    let  the probability measure $\nu_{N, \theta}$ be as in \eqref{eq:nuNtheta}, where $(\xi^i)_{i \in \mbbn}$
    is generated using Latin hypercube sampling.
    Suppose that
    $\xi \sim \text{Uniform}[0,1]^d$.
    Then, the above SLLN
    and \Cref{lem:weakmoment,thm:consistency} with $\Xi = [0,1]^d$ ensure
    that for each random variable $Z \colon [0,1]^d \to \mbbr$ with
    $Z \circ \xi \in \mathcal{L}^{2p}(\Omega)$, we have
    \begin{equation*}
        \rho(\nu_{N, \theta} \circ Z^{-1}) \to \rho(Z(\xi))\asto{N}
        \quad\text{for } \Pb\text{-almost all } \theta \in \Theta.
    \end{equation*}
\end{example}

\section{Epiconvergence of Approximations to Composite Risk Functionals}
\label{sec:epiconvergence}
In this section, we establish our main result,
an almost sure epiconvergence result for composite risk functionals
under consistent pointwise approximations.
Our epiconvergence result is based on two technical assumptions.
The first assumption is based on the consistency
framework developed in
\cite{shapiropaper}.
The second assumption can be viewed as a SLLN,
which is satisfied for probability measures
generated by MC sampling or scrambled net integration.

\begin{assumption}\label{ass:big1}
~
  \begin{enumerate}[nosep, label={(\ref{ass:big1}.\arabic*)}, ref={\ref{ass:big1}.\arabic*}]
        \item \label{ass:big1:item1} The metric space $X$ is separable, and $(\Omega, \mathcal{F},P)$  is
        complete and nonatomic.
        \item
         The map $F: X \times \Xi \to [0, +\infty)$ is random lower semicontinuous.

        \item
            For every $x \in X$, the random variable $F_x(\xi) \coloneqq F(x, \xi)$ is in $ \mathcal{L}^{p+\varepsilon}(\Omega, \mathcal{F}, P)$.

        \item
        The risk measure $\rho \colon \mathcal{L}^p(\Omega, \mathcal{F},P) \to \mbbr$ is convex, and law invariant.
    \end{enumerate}
\end{assumption}

The next assumption formulates a SLLN-type consistency condition,
inspired by the examples considered in \Cref{subsec:examples}.

\begin{assumption}
\label{ass:consistency}
    Let $(\Theta, \mathcal{A}, \Pb)$
    be a probability space.
    For each $\theta \in \Theta$,
    $(\nu_{N, \theta})_{N \in \mbbn}$ is a sequence of probability measures on $(\Xi, \mathcal{B}(\Xi))$.
    For
    each random variable $Z: \Xi \to \mbbr$ with
    $Z \circ \xi \in \mathcal{L}^{p+\varepsilon}(\Omega, \mathcal{F}, P)$, we have
    \begin{equation}\label{assslln}
            \lim_{N \to \infty} \rho(\nu_{N, \theta}\circ Z^{-1}) = \rho(Z(\xi))
            \quad\text{for } \Pb\text{-almost all } \theta \in \Theta.
    \end{equation}

\end{assumption}

Probability measures obtained through MC sampling satisfy \Cref{ass:consistency} with \(\varepsilon = 0\) (see \Cref{exp:mcmeasures}).
Moreover, measures derived from scrambled net integration satisfy \Cref{ass:consistency} with \(\varepsilon > 0\) (see \Cref{exmp:rqmc}).

We now formulate the main result of this section.

 \begin{theorem}[Epiconvergence of Approximations to Risk Functionals]\label{thm:epiconvergence}
    Let \Cref{ass:consistency,ass:big1} be satisfied.
    Then, the sequence of functions
    \begin{equation*}
        x \mapsto \rho(\nu_{N, \theta}\circ F_x^{-1})
    \end{equation*}
    epiconverges to
    \begin{equation*}
        x \mapsto \rho(F_x(\xi))
    \end{equation*}
    for $\Pb$-almost every $\theta \in \Theta$.
\end{theorem}

Before we present our proof of \Cref{thm:epiconvergence},
we state some useful facts and introduce notation.
For two random variables $V$ and $W$
defined on the same probability space, we have
\begin{equation}\label{xgeqy}
    V\geq W\quad\text{almost surely}  \quad  \text{implies} \quad  H[V](t)\geq H[W](t)\quad\text{for all }\quad t \in(0,1).
\end{equation}
Moreover, for each $\alpha \in \mbbr$ we have
\begin{equation}
\label{eq:invscdftranslation}
    H[V + \alpha](t) = H[V](t)+\alpha\quad\text{for all}\quad t \in (0,1).
\end{equation}
For each $x \in X$,
let $H_{N, \theta}[F_x]:(0,1) \to \mbbr$ denote the quantile
function corresponding to the distribution $\nu_{N, \theta} \circ F_x^{-1}$.
By the reasoning given in \Cref{lawinvariance}, we have
    \begin{equation*}
        \rho(\nu_{N, \theta}\circ F_{x}^{-1}) = \rho(H_{N, \theta}[F_{x}](U))\quad\text{for each $x \in X$},
    \end{equation*}
where $U$ is a random variable on $\Omega$ with $U\sim\text{Uniform}(0,1)$.
For the sake of brevity, in the proof below, we denote
the random variable $H_{N, \theta}[F_x](U(\cdot)): \Omega \to
\mathbb{R}$ by $H_{N, \theta}[F_x]$ for each $x \in X$.

Finally, we use Lipschitz regularizations to establish epiconvergence.
Let $d_X$ be the metric on $X$. For $k \in \mbbn$ and a
function $g:X \to (-\infty, +\infty]$, we define its
    $k$-Lipschitz regularization on $X$ by
    \begin{equation*}
        x \mapsto \inf_{x' \in X} \{g(x') + kd_X(x,x')\}.
    \end{equation*}
The $k$-Lipschitz regularization of $g$ is Lipschitz continuous with
constant $k$ if $g$ is proper, lower semicontinuous,
conically minorized, and $k$ is sufficiently large
(see Proposition~3.3 in \cite{Hess1996}).

\begin{proof}[Proof of \Cref{thm:epiconvergence}]
    We begin by proving the liminf part of epiconvergence. As in \cite{Hess2019},
    we establish this part using a standard
    liminf characterization of epiconvergence
    based on $k$-Lipschitz
    regularizations (see, e.g.,
    Proposition~12.1.1 in \cite{Attouch2014}).
    Let $X'$ be a dense, countable subset of $X$. Let $x, x' \in X$.
    Using translation equivariance
    of convex risk measures, and \eqref{eq:invscdftranslation}, we obtain
    \begin{equation*}
        \rho(H_{N, \theta}[F_{x'}]) + kd_X(x,x') = \rho(H_{N, \theta}[F_{x'} + kd_X(x, x')]),
    \end{equation*}
    for each $k \in \mathbb{N}$.
    The function $\xi \mapsto \inf_{x' \in X} \{F_x(\xi) + kd_X(x,x')\}$ is measurable owing to the
    random lower semicontinuity of $F \colon X \times \Xi \to [0, +\infty)$.
    Therefore, monotonicity of $\rho$ and \eqref{xgeqy} yield
    \begin{equation*}
        \inf_{x' \in X}\rho(H_{N, \theta}[F_{x'} + kd_X(x, x')])
        \geq \rho(H_{N, \theta}[\inf_{x' \in X} \{F_{x'} + kd_X(x,x')\}]).
    \end{equation*}

    Now fix $x\in X$ and $k \in \mbbn$. Since for each $x \in X$, the map $F_x(\xi(\cdot))$ is in $\mathcal{L}^{p+\varepsilon}(\Omega)$, and $F$ is nonnegative,
    $\inf_{x' \in X} \{F_{x'}(\xi(\cdot)) + kd_X(x,x')\}$ is also
    in $\mathcal{L}^{p+\varepsilon}(\Omega)$. We can therefore apply \eqref{assslln} to
    obtain the existence of a $\Pb$-negligible set $\Delta_{x,k} \subseteq \Theta$ such that
    for all $\theta \in \Theta \setminus \Delta_{x,k}$,
    \begin{align*}
       \lim_{N\to \infty} \rho(H_{N, \theta}[\inf_{x' \in X} \{F_{x'} + kd_X(x,x')\}])
       = \rho(\inf_{x' \in X} \{F_{x'}(\xi)+kd_X(x,x')\}).
    \end{align*}
    We denote the set $\bigcup_{k\geq 1}\bigcup_{x \in X'} \Delta_{x,k}$
    by $\Delta$. Combining the above derivations yields
    \begin{equation}\label{eq:lipschitz}
        \liminf_{N\to\infty}\inf_{x' \in X}\big\{\rho(H_{N, \theta}[F_{x'}]) + kd_X(x,x')\big\}
        \geq \rho(\inf_{x' \in X} \{F_{x'}(\xi)+kd_X(x,x')\})
    \end{equation}
    for each $x\in X'$ and $\theta \in \Theta \setminus \Delta$.

    Now we extend the convergence statement in \eqref{eq:lipschitz} to all
    $x \in X$.
    Since $\rho$ is real-valued and convex, it is continuous
    (see Proposition~6.7 in \cite{Shapiro2021}). Furthermore,
    for fixed $\xi \in \Xi$, the $k$-Lipschitz regularization of
    $x \mapsto F_x(\xi)$ is Lipschitz continuous.
    Thus, owing to the dominated convergence theorem,
    the right-hand side of \eqref{eq:lipschitz} viewed as a function of $x$ is continuous.
    Moreover, we show that the left-hand side
    of \eqref{eq:lipschitz} viewed as a function of $x$ is upper
    semicontinuous. For each $N \in \mbbn$, $\theta \in \Theta$, define
    $h_{N, \theta}(x) \coloneqq \inf_{x' \in X}\{\rho(H_{N, \theta}[F_{x'}]) + kd_X(x,x')\}$.
    Fix a sequence $(x_\ell)_{\ell \in \mathbb{N}}$
    converging to $x$ in $X$.
    Applying Theorem~9.13 in \cite{DalMaso1993} yields
    $
        h_{N, \theta}(x_\ell) \leq h_{N, \theta}(x) + kd_X(x, x_\ell)
    $
    for each $N \in \mbbn$, $\theta \in \Theta$, and $\ell \in \mathbb{N}$.
    Consequently,
    \begin{equation*}
        \limsup_{\ell \to \infty} \liminf_{N \to \infty}  h_{N, \theta}(x_\ell)
            \leq \liminf_{N \to \infty} h_{N, \theta}(x).
    \end{equation*}
    Hence, $x \mapsto \liminf_{N \to \infty}h_{N, \theta}$ is
    upper semicontinuous.
    Therefore, \eqref{eq:lipschitz} holds for
    each $x \in X$ and $\theta \in \Theta \setminus \Delta$.

    We proceed by demonstrating the identity
    \begin{equation}
        \label{eq:jul22024314}
        \sup_{k \geq 1}\rho(\inf_{x' \in X} \{F_{x'}(\xi)+kd_X(x,x')\}) = \rho(F_x(\xi))
    \end{equation}
    for each $x \in X$.
    Fix $x \in X$ and for each $k \in \mbbn$, define the random variable
    $Y^k_x(\omega) \coloneqq \inf_{x'\in X} \{F_{x'}(\xi(\omega)) + kd_X(x,x')\}$.
    Nonnegativity and lower semicontinuity of $F(\cdot, \xi)$ combined with
    Proposition~3.3 in \cite{Hess1996}  ensure
    $
        0 \leq Y^k_x(\omega) \leq \sup_{k\geq 1} Y^k_x(\omega) = F_x(\xi(\omega))
    $
    for each $\omega \in \Omega$ and $k \in \mbbn$.
    Since $F_x(\xi(\cdot)) \in \mathcal{L}^{p +\varepsilon}(\Omega)$,
    the dominated convergence theorem in conjunction
    with the continuity of $\rho$ yields \eqref{eq:jul22024314}.

    Putting together the pieces, we obtain
    \begin{equation}\label{eq:epiconvergence}
        \sup_{k \geq 1}\liminf_{N \to \infty} \inf_{x' \in X}\big\{\rho(\nu_{N, \theta} \circ F_{x'}^{-1}) + kd_X(x, x')\big\}
        \geq \rho(F_x(\xi)).
    \end{equation}
    for all $\theta \in \Theta \setminus \Delta$ and $x \in X$.
    The set $\Delta$ defined above is $\Pb$-negligible since
    it is the countable union of $\Pb$-negligible sets.
    Therefore, we obtain the liminf part of $\Pb$-almost sure epiconvergence by
    combining \eqref{eq:epiconvergence} with the liminf
    characterization of epiconvergence provided in
    Proposition~12.1.1 in \cite{Attouch2014}.

    Next, we establish the limsup part of almost sure epiconvergence,
    following the arguments from the proof of
    Theorem~A.1 in \cite{milz2023asymptotic}.
    The space $X$ is separable, and the map $x \mapsto \rho(F_x(\xi))$ is
    finite-valued and lower semicontinuous \cite[Thm.\ 3.1]{shapiropaper}. Thus, we can apply
    Lemma~3 in \cite{zervos1999epiconvergence} to obtain the
    existence of countable set $\mathcal{Y} \subset X$ such that for
    each $x \in X$ there exists a sequence
    $(x_\ell)_{\ell \in \mbbn}\subset \mathcal{Y}$ with $x_\ell \to x$
    and $\rho(F_{x_\ell}(\xi)) \to \rho(F_x(\xi))$ as
    $\ell \to \infty$.
    Now, for each $x \in X$, \Cref{ass:consistency} guarantees
    the existence of a  $\Pb$-full set $\Theta_x$ such that
    for each $\theta \in \Theta_x$,
    $\rho(H_{N, \theta}[F_x]) \to \rho(F_x(\xi))$ as $N \to \infty$.
    Since $\mathcal{Y}$ is countable, the set
    $\Theta' \coloneqq \bigcap_{x \in \mathcal{Y}}\Theta_x$ is $\Pb$-full.
    Fix $x \in X$, $\theta \in \Theta'$, and a sequence
    $(x_\ell)_{\ell \in \mathbb{N}} \subset \mathcal{Y}$
    such that $x_\ell \to x$ as $\ell \to \infty$.
    For each $\ell \in \mathbb{N}$, we have
    $\rho(H_{N, \theta}[F_{x_\ell}]) \to \rho(F_{x_\ell}(\xi))$ as $N \to \infty$.
    Next, let $(\delta_{\ell})_{\ell \in \mbbn}$ be a positive sequence with
    $\delta_{\ell} \to 0$ as $\ell \to \infty$.
    For each $\ell \in \mathbb{N}$, we obtain the existence of $k_\ell \in \mathbb{N}$, dependent on $\theta$ and $x_\ell$, such that
    $| \rho(H_{N, \theta}[F_{x_{\ell}}]) - \rho(F_{x_\ell}) | \leq \delta_{\ell}$
    for all $N \geq k_{\ell}$.
    We define the sequence $(K_{\ell})_{\ell \in \mathbb{N}}$ by
    $K_1 \coloneqq k_1$ and $K_{\ell+1} \coloneqq \max\{K_\ell + 1, k_{\ell+1}\}$ for $\ell > 1$.
    By construction, we have $K_\ell \to \infty$ as
    $\ell \to \infty$, and thus
    $x_{K_\ell} \to x$ as $\ell \to \infty$.
    Consequently, $\rho(H_{K_\ell, \theta}[F_{x_{K_\ell}}]) \to \rho(F_x(\xi))$ as $\ell \to \infty$.
    Since we can construct such a sequence $(x_{K_\ell})_{\ell \in \mathbb{N}}$ for arbitrary $x \in X$ and $\theta \in \Theta'$, we obtain the almost sure limsup part of epiconvergence.
\end{proof}\section{Uniform Law of Large Numbers for Composite Risk Functionals}

We establish a uniform law of large numbers. Unlike the epigraphical law of
large numbers, it requires additional assumptions on the decision space, the
integrand, and the probability measures $(\nu_{N,\theta})_{N \in \mbbn}$ in
\Cref{ass:consistency}. In the context of uniform convergence, the additional
hypotheses on the decision space and integrand are standard. We require the probability measures
$(\nu_{N,\theta})_{N \in \mbbn}$ on $\Xi$ be as in \eqref{eq:nuNtheta} and satisfy the
SLLN formulated in \Cref{assumption:slln}, conditions met by MC and scrambled net integration.
Subsequently, we
apply the uniform law of large numbers to demonstrate the consistency of
sample-based approximations of risk-averse stochastic variational inequalities.

The following assumption builds on \Cref{ass:big1}.
\begin{assumption}\label{ass:big2}
~
    \begin{enumerate}[nosep, label={(\ref{ass:big2}.\arabic*)}, ref={\ref{ass:big2}.\arabic*}]
    \item  The metric space $X$ is compact,
    and $(\Omega, \mathcal{F},P)$  is
 complete and nonatomic.
    \item
        The function $F: X \times \Xi \to \mbbr$ is Carath\'eodory mapping.

    \item The map $x \mapsto F_x(\xi) \coloneqq F(x, \xi)$
        from $X$ to $\mathcal{L}^p(\Omega, \mathcal{F}, P)$
        is continuous.

    \item
    The random variable $\kappa \colon \Xi \to \mathbb{R}$
    satisfies $\kappa \circ \xi \in \mathcal{L}^{p + \varepsilon}(\Omega, \mathcal{F}, P)$,
    and $\abs{F(x, \xi)} \leq \kappa(\xi)$ for each $x \in X$ and $\xi \in \Xi$.

    \item The risk measure $\rho \colon \mathcal{L}^p(\Omega, \mathcal{F}, P) \to \mathbb{R}$ is convex, law invariant, and Lipschitz continuous.
    \end{enumerate}
\end{assumption}

Now, we formulate and establish a uniform law of large numbers. We recall the definition of the probability measures $(\nu_{N,\theta})_{N \in \mbbn}$ from \eqref{eq:nuNtheta}.

\begin{theorem}[{Uniform Convergence of Sample-based Risk Functionals}]
\label{thm:uniformlln}
If \Cref{ass:big2,assumption:slln}
hold, then
\begin{equation*}
          \sup_{x \in X} \;\abs{\rho(\nu_{N, \theta} \circ F_x^{-1})-\rho(F_x(\xi))} \to 0 \asto{N}
    \end{equation*}
    for $\Pb$-almost all $\theta \in \Theta$.
\end{theorem}

\begin{proof}
	We adapt the proof of Theorem~9.60 in \cite{Shapiro2021}
    which establishes the standard uniform law of large numbers.
    We note that
    $
        \theta \mapsto \sup_{x \in X}\; \abs{\rho(\nu_{N, \theta} \circ F_x^{-1}) - \rho(F_x(\xi))}
    $
    is measurable because $X$ is separable, and
    $(\theta, x) \mapsto \rho(\nu_{N, \theta} \circ F_x^{-1})-\rho(F_x(\xi))$
    is Carath\'eodory.

    Fix  $y \in X$ and $\delta>0$.
    We demonstrate that there exists a neighborhood $V_{\delta}(y)$
    of $y$ such that for $\Pb$-almost every $\theta \in \Theta$,
    there exists  $N_1^{y,\delta}(\theta) \in \mathbb{N}$
    such that for all $N \geq N_1^{y,\delta}(\theta)$,
    \begin{equation}\label{claim:neighbourhood}
        \sup_{x \in V_{\delta}(y)}
        \abs{\rho(H_{N, \theta}[F_{x}])
        - \rho(H_{N, \theta}[F_{y}])} < \delta.
    \end{equation}
    First, we choose a sequence $(\gamma_k)_{k \in \mbbn}$ with
        $\gamma_k > 0$ and $\gamma_k \to 0$ as $k \to \infty$. We define
        the neighbourhoods $V_k(y) \coloneqq \{x \in X\,|\, d(x,y) < \gamma_k\}$.
        Let
		\begin{align*}
			\Delta_k(\xi) \coloneqq \sup_{x \in V_k(y)}\,
			|F(x,\xi)-F(y,\xi)|^p.
		\end{align*}
		Since $F$ is Carath\'eodory
        and $X$ is separable, $\Delta_k$ is
        measurable. Further,
        $\Delta_k(\xi) \to 0$ as $k \to \infty$
		for every $\xi \in \Xi$.
		Hence the dominated convergence theorem
        when combined with $\Delta_k(\xi) \leq 2^p\kappa(\xi)^p$
        ensures  $\mathbb{E}[\Delta_k(\xi)] \to 0$ as $k \to \infty$.
        Second,  we define the random variable
        \begin{equation*}
    S^{\theta}_{y,N}(\omega) \coloneqq
            \sum_{i=1}^N
            F(y,\xi^i(\theta)) \mathbb{1}_{\mathcal{I}_i}(U(\omega)),
        \end{equation*}
        where $\mathcal{I}_j \coloneqq ((j-1)/N, j/N]$.
        Since $S_{y,N}^{\theta}$
        and $H_{N, \theta}[F_y]$ have the same distribution and
        $\rho$ is law invariant, we have
        $\rho(H_{N, \theta}[F_{x}]) = \rho(S_{x,N}^{\theta})$
        and $\rho(H_{N, \theta}[F_{y}])
            =  \rho(S_{y,N}^{\theta})$.
        Let $L \geq 0$ be the Lipschitz constant of $\rho$. We obtain
        \begin{equation*}
            \sup_{x \in V_k(y)}\,
            | \rho(S_{x,N}^{\theta})-\rho(S_{y,N}^{\theta})|^p
            \leq
            \sup_{x \in V_k(y)} \frac{L^p}{N}\sum_{i=1}^N |F(x,\xi^i(\theta))-F(y,\xi^i(\theta))|^p
            \leq
            \frac{L^p}{N}\sum_{i=1}^N  \Delta_k(\xi^i(\theta)).
        \end{equation*}
        We have $\Delta_k(\xi) \in \mathcal{L}^{1+\varepsilon/p}(\Omega)$ owing to
        $\abs{F(x, \xi)} \leq \kappa(\xi)$ for each $x \in X$, $\xi \in \Xi$,
        and $\kappa \circ \xi \in \mathcal{L}^{p + \varepsilon}(\Omega)$.
		Hence, the SLLN in \Cref{assumption:slln} guarantees the existence of a $\Pb$-full set $\Theta_1^{y,k}$ such that
        \begin{equation}\label{eq:conv}
            \frac{1}{N}\sum_{i=1}^N  \Delta_k(\xi^i(\theta))
            \to
            \mathbb{E}[\Delta_k(\xi)]\asto{N}\quad\text{for each} \quad \theta \in \Theta_1^{y,k}.
        \end{equation}
        Since $\mathbb{E}[\Delta_k(\xi)] \to 0$ as $k \to \infty$,
        we obtain the existence of $K_{\delta} \in \mbbn$ such that
        for all $k \geq K_{\delta}$, we have
        $L^p\mathbb{E}[\Delta_k(\xi)] < \delta^p/2$.
        Combined with \eqref{eq:conv}, we obtain
        for each $\theta \in
        \Theta_1^{y, K_{\delta}}$, the existence of
        $N_1^{y, {\delta}}(\theta) \in \mathbb{N}$ such that for all
        $N \geq N_1^{y, \delta}(\theta)$,
        $
            (L^p/N)\sum_{i=1}^N \Delta_{K_{\delta}}({\xi^i(\theta)})
            \leq \delta^p.
        $
         This concludes the verification of \eqref{claim:neighbourhood}.

        Next, owing to
        \Cref{assumption:slln,lem:weakmoment}, and
        $F_y(\xi) \in \mathcal{L}^{p + \varepsilon}(\Omega)$,
        we obtain the existence of a $\Pb$-full set $\Theta^{y}_2$
        such that for each $\theta \in \Theta^{y}_2$,
        there exists $N_2^{y,\delta}(\theta) \in \mbbn$ with
        \begin{equation*}
            \abs{\rho(H_{N, \theta}[F_{y}])-\rho(F_{y}(\xi))}
                < \delta
        \end{equation*}
        for all $N \geq N_2^{y,\delta}(\theta)$.
        Furthermore, the continuity of the map $x \mapsto \rho(F_x(\xi))$ guarantees the existence of a neighborhood $W_{\delta}(y)$ of $y$ such that
        \begin{equation*}
            \sup_{x \in W_{\delta}(y)}
            \abs{\rho(F_x(\xi))-\rho(F_{y}(\xi))} < \delta.
        \end{equation*}
        We define
        $\Lambda_{\delta}(y) \coloneqq
            W_{\delta}(y) \cap V_{K_\delta}(y)$,
        and $N^{y, \delta}_*(\theta) \coloneqq
        \max\{N^{y,\delta}_1(\theta), N_2^{y,{\delta}}(\theta)\}$.

        Since $X$ is compact, the cover
        $\bigcup_{y\in X}\Lambda_{\delta}(y)$ has a finite subcover
        with finitely many centers
        $y_1, \ldots, y_{m_{\delta}}$. For each
        $x \in X$, we choose $j\in \{1, \ldots, m_{\delta}\}$ such that
        $x \in \Lambda_{\delta}(y_j)$. Then
        \begin{align*}
            \abs{\rho(H_{N, \theta}[F_x])-\rho(F_x(\xi))} &\leq
            \abs{\rho(H_{N, \theta}[F_x])-\rho(H_{N, \theta}[F_{y_j}])} \\
            &\quad +\abs{\rho(H_{N, \theta}[F_{y_j}]) -\rho(F_{y_j}(\xi))} +
            \abs{\rho(F_{y_j}(\xi)) -\rho(F_x(\xi))}
            \leq 3\delta,
        \end{align*}
        for all $N \geq \max\{N^{y_1, \delta}_{*}(\theta), \ldots, N^{y_m, \delta}_{*}(\theta)\}$
        and $\theta \in \bigcap_{1\leq j\leq m_{\delta}}
        (\Theta^{y_j,K_{\delta}}_1 \cap \Theta^{y_j}_2)$.
        Since the intersection of the latter set over$\delta \in \mathbb{Q}_{>0}$
        is a countable intersection of $\Pb$-full sets, we obtain the convergence statement.
\end{proof}

\subsection{Risk-Averse Stochastic Variational Inequalities}
We use our uniform law of large numbers, \Cref{thm:uniformlln}, to demonstrate the consistency of sample-based approximations to  risk-averse stochastic variational
inequalities.
Let $r \in \mathbb{N}$ and  let $\psi: \mbbr^r \to (-\infty, +\infty]$ be a proper, closed,
and convex function with domain $X$. For each $j \in \{1, \ldots, r\}$,
let $\rho_j: \mathcal{L}^p(\Omega, \mathcal{F},P) \to \mathbb{R}$
be a law invariant risk measure
and $F^j:\mbbr^r \times \Xi \to \mathbb{R}$ are maps.
We consider the risk-averse stochastic variational inequality
\begin{equation}\label{eq:simplevi}
     0 \in \Big[\rho_{1}(F^1_x(\xi)), \ldots, \rho_{r}(F^r_x(\xi))\Big]
        + {\partial\psi(x)}.
\end{equation}
For fixed $\theta \in \Theta$
and $N \in \mathbb{N}$, we define
an empirical estimate of \eqref{eq:simplevi}
by
\begin{equation}\label{eq:empviestimates}
    0 \in \Big[\rho_{1}(\nu_{N, \theta}\circ (F_x^1)^{-1}), \ldots, \rho_{r}(\nu_{N, \theta}\circ (F_x^r)^{-1})\Big]
            + {\partial\psi(x)}.
\end{equation}
If a Uniform Law of Large Numbers in the sense
of \Cref{thm:uniformlln} holds for each $j \in \{1, \ldots, r\}$,
then solutions to \eqref{eq:empviestimates} consistently
approximate solutions of \eqref{eq:simplevi}
as we demonstrate next.

\begin{lemma}[Risk-Averse Stochastic Variational Inequalities]\label{exp:vareq}
Let $\psi: \mbbr^r \to (-\infty, +\infty]$ be a proper, closed,
and convex function with compact domain $X$.
Let \Cref{assumption:slln} hold.
Suppose that for each pair
$(\rho_j, F^j)$, $j \in \{1, \ldots, r\}$, \Cref{ass:big2} is satisfied.
    If for $\Pb$-almost all $\theta \in \Theta$
    there exists $(x_{N,\theta})_{N \in \mbbn} \subset X$ such that
    $x_{N,\theta}$ solves \eqref{eq:empviestimates}
    for all $N \in \mbbn$, then for
    $\Pb$-almost all $\theta \in \Theta$, each accumulation point
    of
    $(x_{N,\theta})_{N \in \mbbn}$
    solves \eqref{eq:simplevi}.
    \end{lemma}
    \begin{proof}
    For $\Pb$-almost all $\theta \in \Theta$,
    let $x_\theta$ be an accumulation point of $(x_{N,\theta})_{N \in \mbbn}$,
    and let us denote the subsequence converging to $x$ by $(x_{N,\theta})_{N \in \mbbn}$ again.
    Because $x \mapsto \rho(F^j_x(\xi))$ is continuous,
    \Cref{thm:uniformlln} yields
    $\rho_{1}(\nu_{N, \theta}\circ (F_{x_{N, \theta}}^j)^{-1}) \to \rho(F_{x_{\theta}}^j(\xi))$
    as $N \to \infty$ for all
    $j \in \{1, \ldots, r\}$
    and $\mathbb{P}$-almost all $\theta \in \Theta$.
    Hence $x_\theta$ solves \eqref{eq:simplevi}.
\end{proof}

\section{Numerical Experiments}
\label{sec:numericalexperiements}
This section's objective is to numerically demonstrate that RQMC sample-based
approximations of risk-averse stochastic programs potentially reduce
RMSE and bias when compared with the MC sample-based scheme.
We approximate the optimal values of
\begin{equation}\label{eq:opnum}
    \min_{x \in X} \rho(F_x(\xi))
\end{equation}
using its sample-based counterparts
\begin{equation}\label{eq:samplebasednum}
    \min_{x \in X} \rho(\nu_{N, \theta} \circ F_x^{-1}),
\end{equation}
where the sequences $(\nu_{N, \theta})_{N \in \mathbb{N}}$ satisfy
\Cref{ass:consistency}. Throughout the section, $X$ is a compact subset in $\mathbb{R}^d$.
In this case, our epigraphical law of large number
when combined with standard properties
of epiconvergence
(see, e.g., Theorem~12.1.1 in \cite{Attouch2014}) provides conditions sufficient for almost sure consistency of the optimal value of
\eqref{eq:samplebasednum} towards that of
\eqref{eq:opnum}.
We refer to \eqref{eq:samplebasednum} as
sample-based problem or approximation problem.
The samples \(\xi^1, \xi^2, \ldots\) are generated using (RQ)MC methods.
We solve problems of the form \eqref{eq:samplebasednum} for increasing
sample sizes \(N\) and compute the RMSE and
bias. We run simulations for MC, scrambled Sobol', Halton,
and Latin Hypercube samples.
For one instance of \eqref{eq:opnum},
$\xi$ is a Gaussian random vector.
In addition to the aforementioned sample-based schemes, we employ Sobol' sequences combined with principal component analysis (PCA) as outlined in \cite[p.\ 594]{Heitsch2016}.

For our experiments, we use \texttt{Julia} along with
the library \texttt{JuMP} \cite{Lubin2023} to model problems of type \eqref{eq:samplebasednum},
\texttt{SCS} \cite{odonoghue:21} to solve
conic problems,
\texttt{HiGHS} \cite{MR3773090} to solve linear programs,
and \texttt{scipy}'s QMC engine \cite{Roy2023} to generate scrambled samples.
Besides the simulation output, our graphical illustrations contain further simulation details. Specifically, $M$ denotes the number of replications, and
    $d$ is the dimension of the decision variable $x$.
    The parameter, ref, indicates that the optimal value of \eqref{op} was estimated by averaging
    those of $M$ approximation problems, each based on \(2^{\text{ref}}\) samples of scrambled Sobol'
    sequences.
The computer code and simulation output is archived at \cite{olena_melnikov_2024_13227277}.

\subsection{Portfolio Optimization}\label{portfolio}
We consider Problem~4.11 in \cite{LanNemSha2012}, which is given by
\begin{align}\label{eq:nemproblem}
   \min_{x \in \mathbb{R}^d}\; \mathrm{CVaR}_{\beta}[-\xi^T x]\quad \text{s.t.}\quad
   \sum_{i=1}^d x_i = 1,\; \mu^T x \geq R,\; x \geq 0,
\end{align}
where $$\mathrm{CVaR}_\beta[V]
\coloneqq \inf_{t \in \mathbb{R}}\,
\{t + (1/(1-\beta)) \mathbb{E}[\max\{V-t,0\}]\}$$
is the conditional value-at-risk
with level $\beta \in (0,1)$
for an integrable random variable $V$. Moreover,
$\mu \in \mathbb{R}^d$, and $R \in \mathbb{R}$.
We perform experiments for two different distributions of $\xi$, as specified
in the following two sections.

\subsubsection{Normally Distributed Samples}\label{normal_section}
First, we choose a normally distributed
random vector $\xi$ with mean $\mu$ and covariance matrix $\Sigma$.
In this case, \eqref{eq:nemproblem} reduces
to a quadratic conic problem
(for details, see \cite{LanNemSha2012}). Hence, we can
compare the SAA optimal values with the exact ones.
We generate $\mu$ and $\Sigma$ as in \cite{LanNemSha2012}.
Specifically, we generate $\mu$ from
$\text{Uniform}[0.9,1.2]^d$ and $Q$
from $\text{Uniform}[0,0.1]^{d\times d}$, and set $\Sigma = QQ^T$.
The simulation output and parameter choices for Problem
\eqref{eq:nemproblem} with normally distributed $\xi$
are depicted in \Cref{fig:normal}.

\begin{figure}[!t]
    \centering
    \begin{minipage}[b]{0.49\textwidth}
        \centering
        \includegraphics[width=\textwidth]{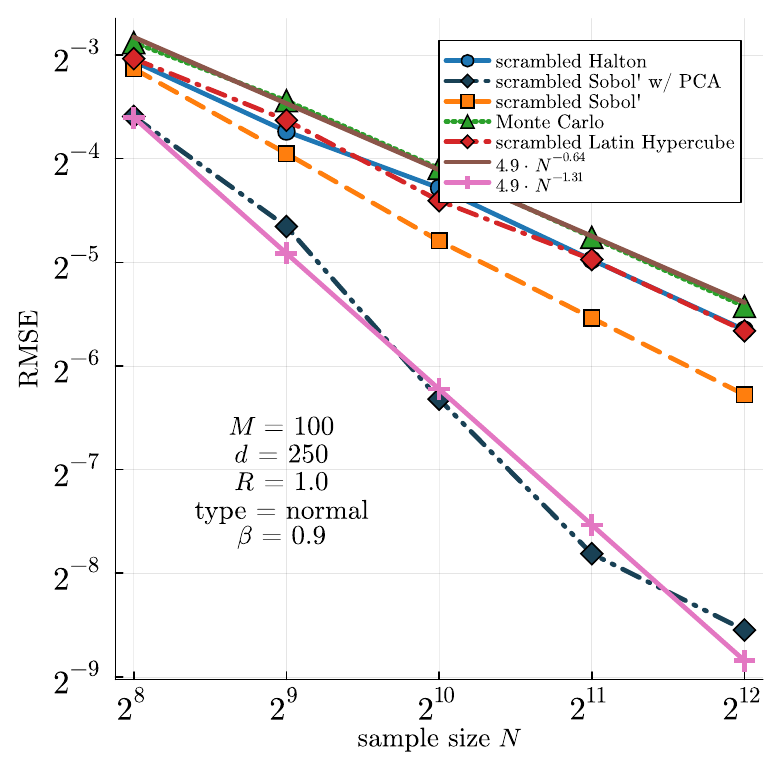}
        \label{fig:subfig11}
    \end{minipage}
    \hfill
    \begin{minipage}[b]{0.49\textwidth}
        \centering
        \includegraphics[width=\textwidth]{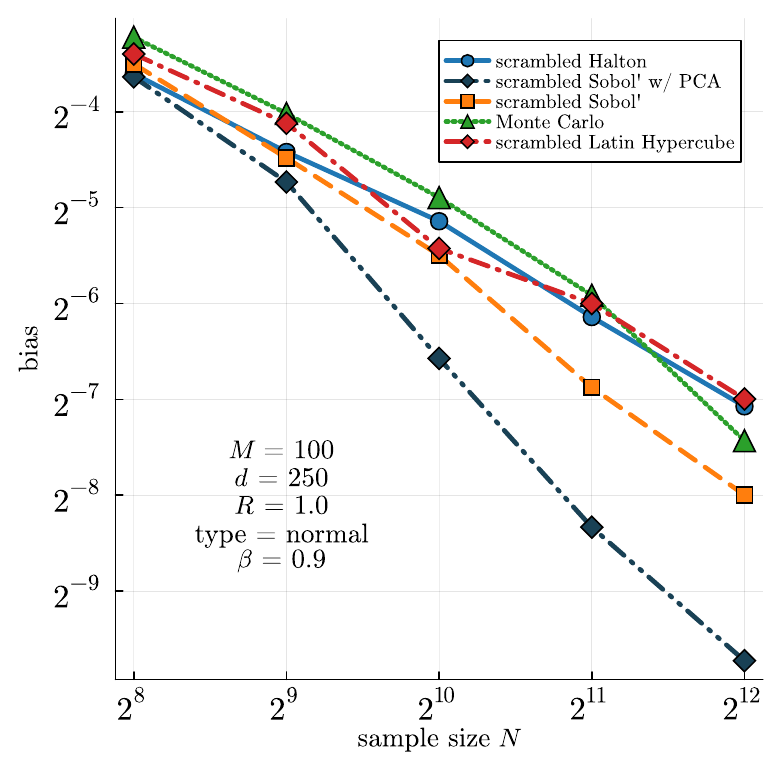}
        \label{fig:subfig12}
    \end{minipage}
    \vspace{0.5cm}
    \begin{minipage}[b]{0.49\textwidth}
        \centering
        \includegraphics[width=\textwidth]{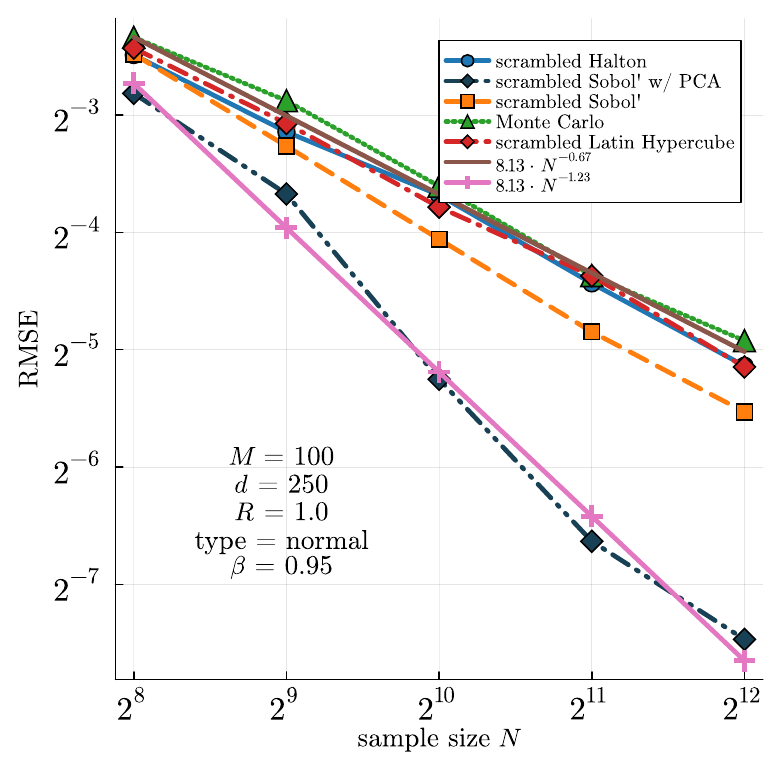}
        \label{fig:subfig13}
    \end{minipage}
    \hfill
    \begin{minipage}[b]{0.49\textwidth}
        \centering
        \includegraphics[width=\textwidth]{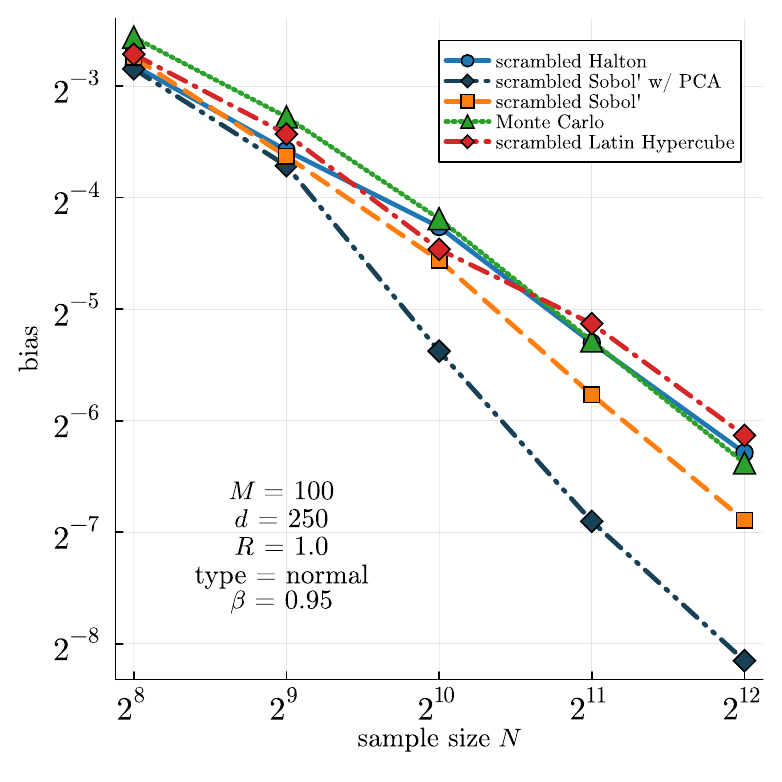}
        \label{fig:subfig14}
    \end{minipage}
    \caption{RMSE and bias  for the sample-based approximation of Problem \eqref{eq:nemproblem}. The experiments were conducted for a  normally distributed
random vector $\xi$ as described in \Cref{normal_section}, indicated by "type=normal." The top two panels correspond to $\beta = 0.9$, and the bottom two panels correspond to $\beta = 0.95$. The value $R$ is the constraint parameter given in  \eqref{eq:nemproblem}. The plots depicting the RMSE also include the least squares fit for the MC
    and scrambled Sobol' experiments.}
    \label{fig:normal}
\end{figure}

\subsubsection{Linearly Transformed Uniform Samples}\label{uniform_section}
For our second instance, we consider a distribution of $\xi$
defined in \cite[p.\ 471]{koivu2004variance}. Specifically, we define the random vector $\xi$ by
$
    \xi \coloneqq \mu + \sqrt{12}Q\zeta
$,
where $\zeta \sim \text{Uniform}[-1/2,1/2]^d$, and $Q$ and $\mu$
are generated as in \Cref{normal_section}.
Unlike the case for normally distributed $\xi$,
the problem does not reduce to a deterministic problem.
Thus, we approximate the optimal value of \eqref{eq:nemproblem}
using its sample-based counterparts.
Specifically, we solve $100$ sample-based approximations of
\eqref{eq:nemproblem}
and average their optimal values. Each of these approximations
is based on a realization of a scrambled Sobol' sequence with
a sufficiently large sample size.
The experiment outcomes and parameter choices for Problem
\eqref{eq:nemproblem} with $\xi$ defined as in this section
are shown in \Cref{fig:uniform}.

\begin{figure}[!t]
    \centering
    \begin{minipage}[b]{0.49\textwidth}
        \centering
        \includegraphics[width=\textwidth]{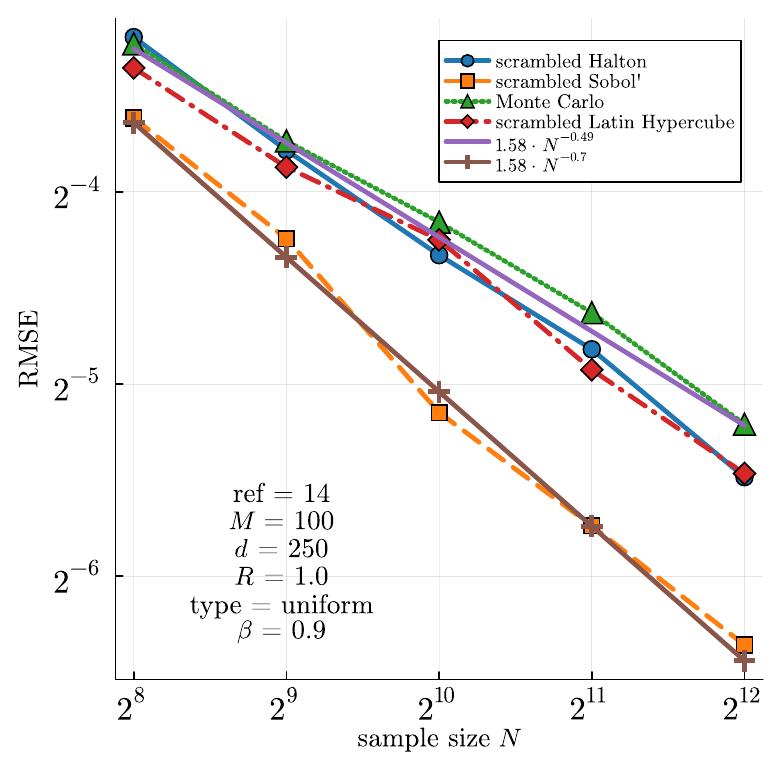}
        \label{fig:subfig21}
    \end{minipage}
    \hfill
    \begin{minipage}[b]{0.49\textwidth}
        \centering
        \includegraphics[width=\textwidth]{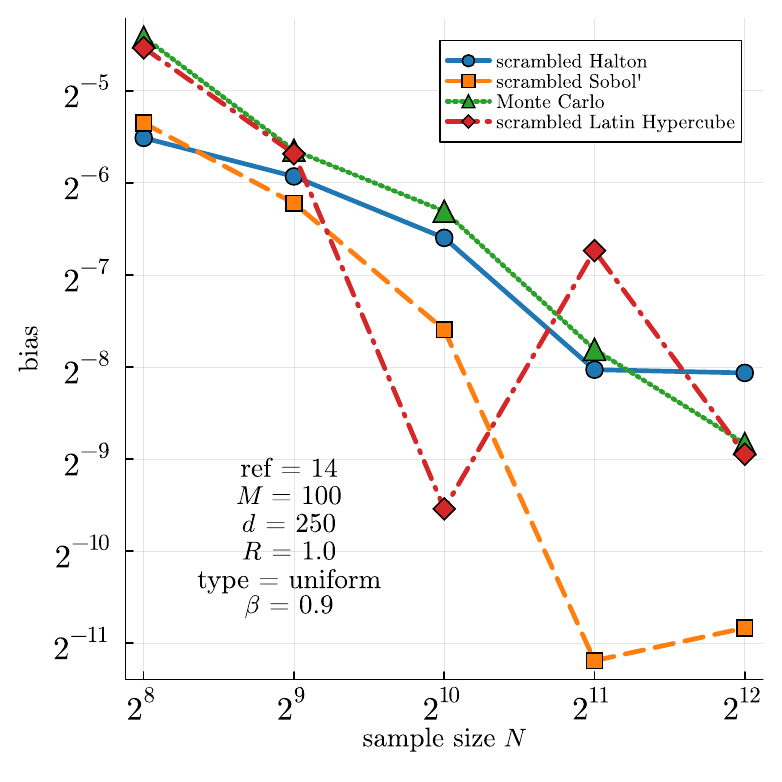}
        \label{fig:subfig22}
    \end{minipage}
    \vspace{0.5cm}
    \begin{minipage}[b]{0.49\textwidth}
        \centering
        \includegraphics[width=\textwidth]{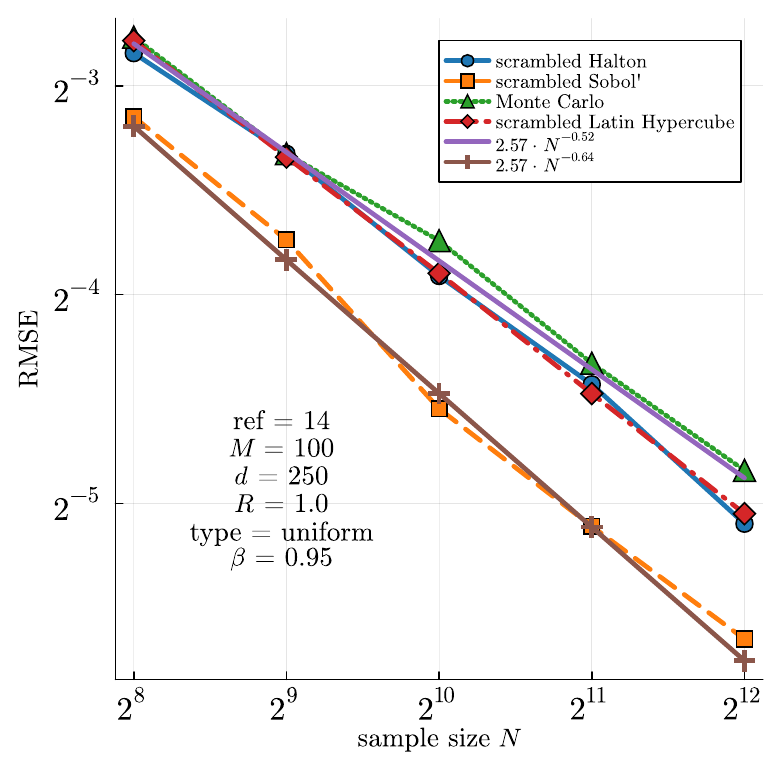}
        \label{fig:subfig23}
    \end{minipage}
    \hfill
    \begin{minipage}[b]{0.49\textwidth}
        \centering
        \includegraphics[width=\textwidth]{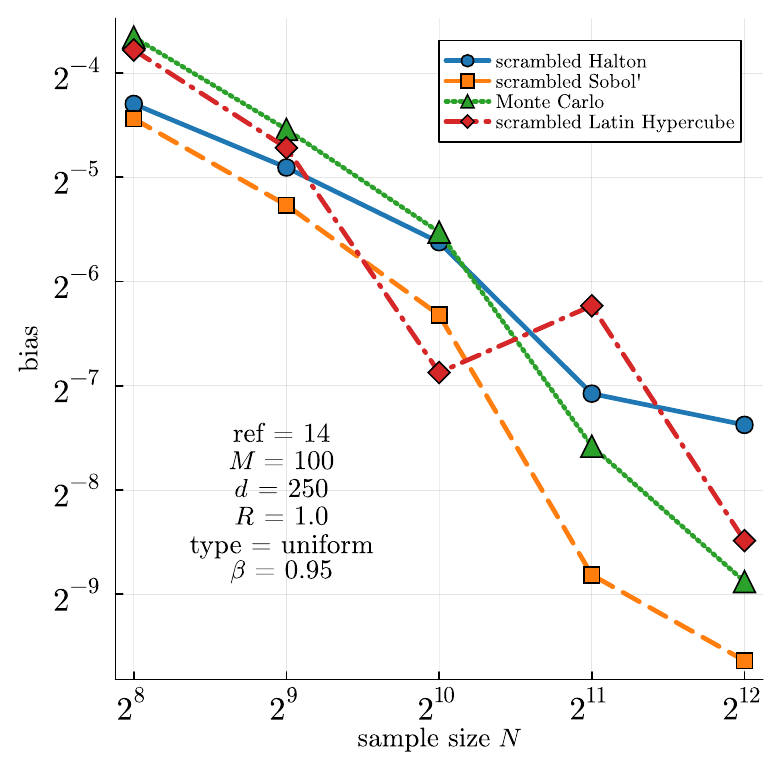}
        \label{fig:subfig24}
    \end{minipage}
    \caption{RMSE and bias  for the sample-based approximation of Problem \eqref{eq:nemproblem}. The experiments were conducted for $\xi$ distributed as described in \Cref{uniform_section}, indicated by "type=uniform."
    The value $R$ is the constraint parameter given in  \eqref{eq:nemproblem}. The top two panels correspond to $\beta = 0.9$, and the bottom two panels correspond to $\beta = 0.95$. The plots depicting the RMSE also include the least squares fit for the MC and scrambled Sobol' experiments.}
    \label{fig:uniform}
\end{figure}

\subsection{A Utility Problem}
\label{sec:twostage}
The second test problem is a slight alteration of a problem
in \cite{Zhang2021}. Let $(T,v,e)$ be a random
vector supported on $\mbbr^{md+ 2m}$,
and let $c \in \mbbr^d$. We consider
\begin{align}\label{op}
   \min_{x\in \mbbr^d}\; c^Tx+\mathrm{CVaR}_{\beta}[Q(x, (T,v,e))]&\quad\text{s.t.}\quad0\leq x_i \leq 1,\quad i = 1, \ldots, d,
\end{align}
where
\begin{align*}
    Q(x,(T,v,e)) \coloneqq \min_{y \in \mbbr^m_{+}} \{y^Te \,
    \colon \quad y\geq v-Tx \}.
\end{align*}
The triple $(T,v,e)$ is
a reshaped and transformed random vector $\xi \sim \text{Uniform}[0,1]^{md+2m}$.
In more detail, we generate samples
$(T_1, v_1, e_1), (T_2, v_2, e_2), \ldots$ from random vectors
$\xi^i \sim \text{Uniform}[0,1]^{md+2m}$ and $c$ as follows.
We choose $c \sim \text{Uniform}[0,1]^d$.
We generate the random vectors $\xi^1, \xi^2, \ldots, \xi^N\in
\mbbr^{md+2m}$ using (RQ)MC methods.
We reshape each $\xi^i$ into a tuple $(T_i, v_i, e_i)$ and
transform the tuple such that for each $i \in \{1, \ldots, N\}$,
the entries of $T_i$ are in $[0.5, 1.0]$,
those of  $v_i$ are in $[50,100]$,
and those of $e_i$ are in $[2,4]$.
As in \Cref{uniform_section}, we cannot solve
\eqref{op} exactly. Thus, we will approximate the optimal value of
\eqref{op} as in \Cref{uniform_section}.
The experiment outcomes and parameter choices for Problem
\eqref{op} are depicted in \Cref{fig:2stage}.

\begin{figure}[!t]
    \centering
    \begin{minipage}[b]{0.49\textwidth}
        \centering
        \includegraphics[width=\textwidth]{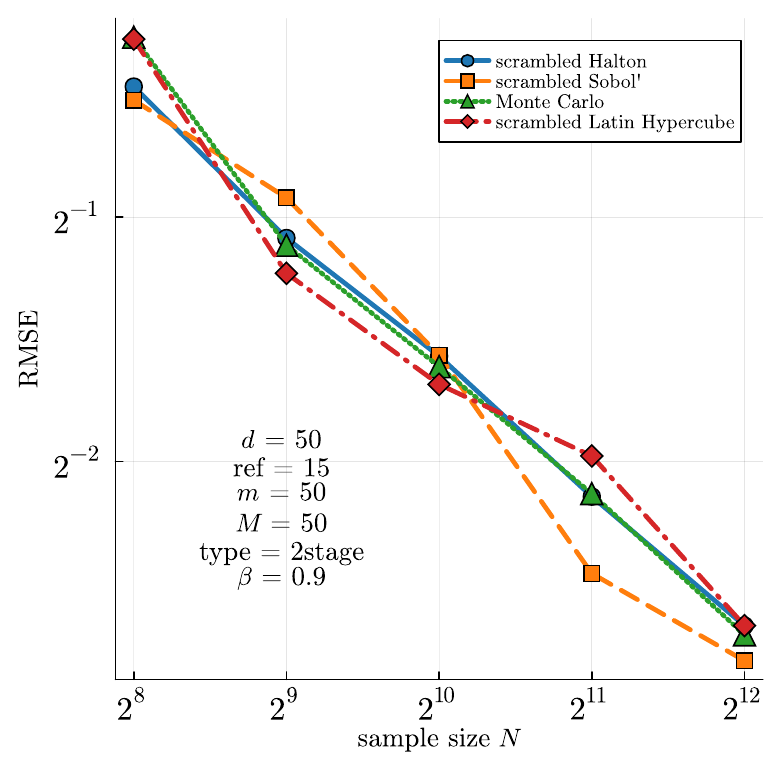}
        \label{fig:subfig31}
    \end{minipage}
    \hfill
    \begin{minipage}[b]{0.49\textwidth}
        \centering
        \includegraphics[width=\textwidth]{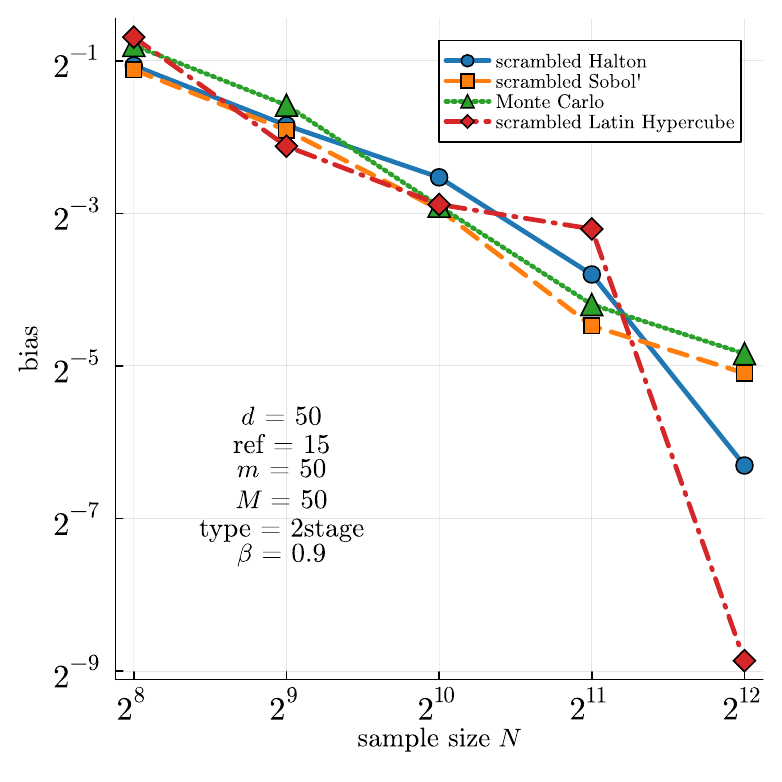}
        \label{fig:subfig32}
    \end{minipage}
    \vspace{0.5cm}
    \begin{minipage}[b]{0.49\textwidth}
        \centering
        \includegraphics[width=\textwidth]{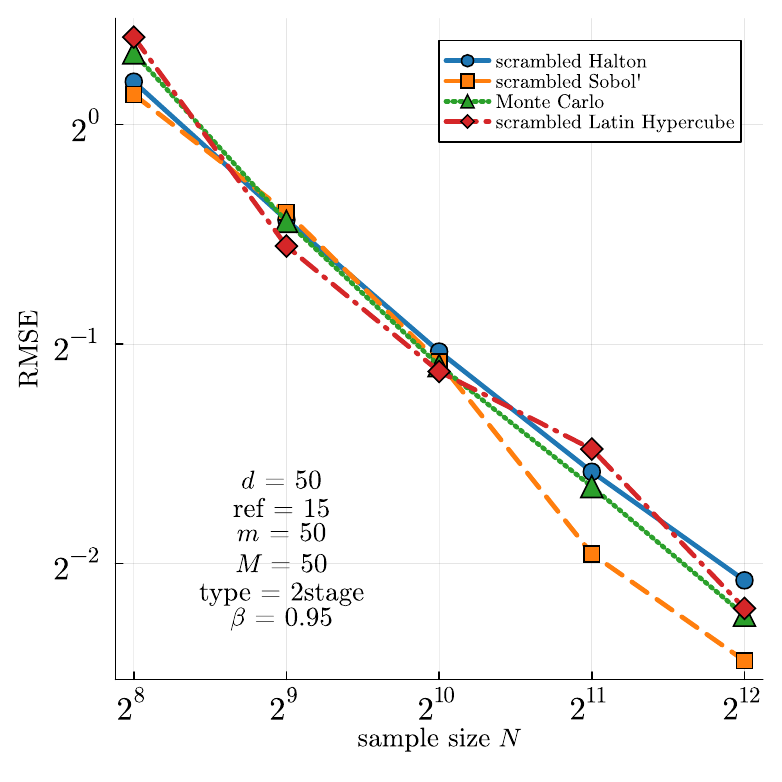}
        \label{fig:subfig33}
    \end{minipage}
    \hfill
    \begin{minipage}[b]{0.49\textwidth}
        \centering
        \includegraphics[width=\textwidth]{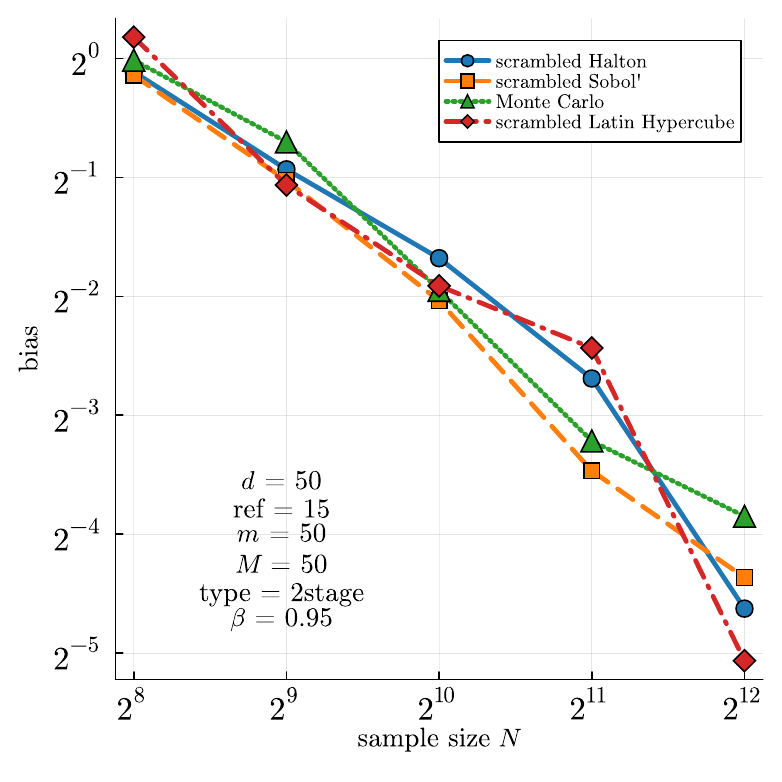}
        \label{fig:subfig34}
    \end{minipage}
    \caption{RMSE and bias  for the sample-based approximation of Problem \eqref{op}, indicated by "type=2stage." The top two panels correspond to $\beta = 0.9$, and the bottom two panels correspond to $\beta = 0.95$. The number $m$ is the dimension of $v$ and
    $e$ in the random element $(T,v,e) \in \mbbr^{(m\times d) + m + m}$.} \label{fig:2stage}
\end{figure}

\section{Discussion}

Motivated by a SLLN for scrambled net integration
\cite{Owen2021}, we have demonstrated the
epiconvergence and uniform convergence
for sample-based approximations of composite
risk functionals, provided that the sample-based
approximations satisfy a pointwise SLLN\@.
Besides MC sampling and
scrambled net integration, Latin hypercube sampling
yields epiconvergent and uniform approximations.

For normally distributed random inputs,
our numerical simulations show that scrambled Sobol'
integration can significantly reduce bias and RMSE
when combined with dimension reduction via PCA\@.
For the two-stage problem considered
in \Cref{sec:twostage}, the simulation output indicates that the MC sampling and the scrambled Sobol sequences have
similar RQMC and bias. In light of the
theoretical considerations, discussions, and
numerical simulations in \cite{Heitsch2016},
this could be a result of the potentially high
effective dimension.
Considering all of our numerical simulations and performance metrics,
scrambled Sobol' integration is always at least as good as MC sampling.

\subsection*{Data Availability Statement}
The computer code and the simulation output is
archived at \url{https://doi.org/10.5281/zenodo.13227277}.

\begin{footnotesize}
		\bibliography{RQMC-Melnikov-Milz.bbl}
	\end{footnotesize}
\end{document}